\documentclass[11pt]{article}
\usepackage{verbatim}
\usepackage{amsthm}
\usepackage{mathtools, thmtools}
\usepackage{caption}
\usepackage{multirow}
\usepackage{array}
\usepackage{empheq}
\usepackage{tabularx,float}
\usepackage{xcolor}
\usepackage[english]{babel}
\usepackage{blindtext}
\usepackage{longtable}
\usepackage[numbers]{natbib}
\usepackage{titlesec}
\usepackage[section]{placeins}
\usepackage[linesnumbered,ruled,vlined]{algorithm2e}
\usepackage{amsfonts,amssymb,amsmath,bm,a4,color,amsthm}
\usepackage{hyperref}
\hypersetup{hypertexnames=false} 

\renewcommand{\Re}{{\textup{Re}}}
\newtheorem{theorem}{Theorem}

\newtheorem{corollary}{Corollary}
\newtheorem{lemma}{Lemma}
\theoremstyle{remark}

\declaretheoremstyle[bodyfont=\normalfont\itshape, notefont=\bfseries, notebraces={}{}, headpunct={}, postheadspace=1mm]{mystyle}

\begin{document}

\title{Explicit estimates for the Riemann zeta function close to the $1$-line}
\author{Michaela Cully-Hugill\footnote{\textbf{Address:} School of Science, University of New South Wales (Canberra), ACT, Australia \\ \indent\textbf{Email:} m.cully-hugill@unsw.edu.au}\, and Nicol Leong\footnote{\textbf{Address:} School of Science, University of New South Wales (Canberra), ACT, Australia \\ \indent \textbf{Email:} nicol.leong@unsw.edu.au} }
\date{}

\maketitle

\begin{abstract}
We provide explicit upper bounds of the order $\log t/\log\log t$ for $|\zeta'(s)/\zeta(s)|$ and $|1/\zeta(s)|$ when $\sigma$ is close to $1$. These improve existing bounds for $\zeta(s)$ on the $1$-line. 
\end{abstract}

\section{Introduction}

In the study of the Riemann zeta function $\zeta(s)$ and related functions, one often requires upper bounds for the logarithmic derivative of $\zeta(s)$ and the reciprocal of $\zeta(s)$. The Chebyshev function $\psi(x)=\sum_{n\le x}\Lambda(n)$, where $\Lambda(n)$ is the von Mangoldt function, or $M(x)=\sum_{n\le x} \mu (n)$, where $\mu(n)$ is the M\"{o}bius function are particularly relevant examples of where these estimates are useful. To estimate either with Perron's formula, we use the fact that 
\begin{equation}\label{dirichlet series}
\sum_{n\geq 1}\frac{\Lambda(n)}{n^s} = -\frac{\zeta'(s)}{\zeta(s)} \qquad \text{and} \qquad \sum_{n\geq 1}\frac{\mu (n)}{n^s} = \frac{1}{\zeta(s)}
\end{equation} for $\Re(s)>1$ (see e.g. \cite{trudgian2015explicit}, \cite{chalker}, \cite{Dudek_16}, \cite[(3.14)]{titchmarsh1986theory} for more detail).

Trudgian \cite{trudgian2015explicit} gave explicit upper bounds of the order $O(\log t)$ for both functions in \eqref{dirichlet series} (see also \cite[Theorem~6.7]{M_V_73}). These estimates are especially notable because they are valid inside the critical strip ($0<\Re(s)<1$), but only within known zero-free regions. Because of this, the constants had to be sufficiently large to bound the functions as $\zeta(s)\rightarrow 0$, which would happen as $s$ approaches the boundary of the zero-free region. As such, the problem of improving these explicit constants is extremely non-trivial, and explicit estimates heavily depend on explicit zero-free regions.

The recent work of Yang \cite{yang2023explicit} establishes a Littlewood-type zero-free region (see Lemma \ref{littlewood_zero_free_reg}), which is the largest known region for $\exp(209)\le t\le \exp(5\cdot 10^5)$. With this, Hiary, Leong and Yang \cite{hiaryleongyangArxiv} gave bounds on $|\zeta'(s)/\zeta(s)|$ and $|1/\zeta(s)|$ of the order $\log t/\log\log t$ on the $1$-line, which are the first explicit bounds of this type. The goal of this article is to improve and extend their results inside the critical strip.

\section{Main results}

The following four corollaries are the new estimates we prove in Sections \ref{mainsec1} and \ref{mainsec2}. They are the first explicit bounds of this order inside the critical strip, and also improve current bounds at $\sigma = 1$. On the $1$-line, Corollary \ref{cor:zeta'/zeta_main_lem1} improves \cite[Theorem~2]{hiaryleongyangArxiv}, and Corollary \ref{cor:1/zeta at 1} improves \cite[Theorem~2]{hiaryleongyangArxiv} for $t\ge 3$, and \cite[Proposition~A.2]{carneiro2022optimality} for $t\geq 222116$. 

Throughout this article, we denote the height to which the Riemann Hypothesis has been verified \cite{PlattTrudgianRH} by $H:=3\,000\,175\,332\,800$.

\begin{corollary}\label{cor:zeta'/zeta_main_lem}
    For $t\geq t_0 \ge e^e$ and $$\sigma\geq 1 - \frac{\log\log t}{W\log t},$$ we have 
    \begin{equation*}
        \left| \frac{\zeta'(\sigma +it)}{\zeta(\sigma+it)}\right| \le R_1\frac{\log t}{\log\log t},
    \end{equation*}
    where $(W,R_1) = (22,5471)$ is valid for $t_0=e^e$, and other values are given in Table \ref{tab:log-deriv-zeta}.
\end{corollary}

\begin{corollary}\label{cor:zeta'/zeta_main_lem1}
    For $t\geq t_0$ and $\sigma\geq 1$, we have
        \begin{equation*}
        \left| \frac{\zeta'(\sigma +it)}{\zeta(\sigma+it)}\right| \le K_1\frac{\log t}{\log\log t},
    \end{equation*}
    where $K_1 = 113.3$ is valid for $t_0=500$, and other values are given in Table \ref{tab:log-deriv-zeta1}.
\end{corollary}

\begin{corollary}\label{cor:1/zeta}
    For $t\geq t_0\geq e^e$ and $$\sigma\geq 1-\frac{\log\log t}{W\log t},$$ we have
    \begin{equation*}
        \left| \frac{1}{\zeta(\sigma+it)}\right| \le R_2\frac{\log t}{\log\log t},
    \end{equation*}
    where $(W, R_2) = (22, 3438)$ is valid for $t_0 = 500$, and other values are given in Table \ref{tab:1/zeta}.
\end{corollary}

\begin{corollary}\label{cor:1/zeta at 1}
    For $t\geq t_0\geq e^e$ and $\sigma\geq 1$ we have
    \begin{equation*}
        \left| \frac{1}{\zeta(\sigma+it)}\right| \le K_2\frac{\log t}{\log\log t},
    \end{equation*}
   where $K_2 = 107.7$ is valid for $t_0=500$, and other values are given in Table \ref{tab:1/zeta at 1}. 
   
   Furthermore, with the restriction $\sigma=1$, the range of $t$ can be extended. In other words, for $t\geq 3$ we have
    \begin{equation*}
        \left| \frac{1}{\zeta(1+it)}\right| \le 107.7\frac{\log t}{\log\log t}.
    \end{equation*}
\end{corollary}

\section{Some preliminaries}
This section contains the lemmas and theorems used to prove our main results. The first two lemmas, which form the basis of our method, are built upon a theorem of Hadamard, Borel and Carath\'eodory, sometimes also referred to as the Borel--Carath\'{e}odory theorem.
This theorem enables us to deduce an upper bound for the modulus of a function and its derivatives on a circle, using bounds for its real part on a larger concentric circle.

\begin{theorem}\cite[Theorem~4]{hiaryleongyangArxiv}[Borel--Carath\'{e}odory]\label{borelcaratheodory}
Let $s_0$ be a complex number, and let $R$ be a positive number possibly dependent on $s_0$. 
Suppose that the function $f(s)$ is analytic in a region containing
the disc $|s-s_0|\le R$. Let $M$ denote the maximum of $\textup{Re}\, f(s)$ on the boundary $|s-s_0|=R$.
Then, for any $r\in (0,R)$ and any $s$ such that $|s-s_0|\le r$,
\begin{equation*}
|f(s)| \le \frac{2r}{R-r}M + \frac{R+r}{R-r}|f(s_0)|.
\end{equation*}
If in addition $f(s_0)=0$, then 
for any $r\in (0,R)$ and any $s$ such that $|s-s_0|\le r$,
\begin{equation*}
|f'(s)| \le \frac{2R}{(R-r)^2}M.
\end{equation*}
\end{theorem}

\begin{lemma}\cite[Lemma~5]{hiaryleongyangArxiv}\label{log_deriv_zeta_lem1}
Let $s_0$ be a complex number, and let $r$ and $\alpha$ be positive numbers (possibly depending on $s_0$) such that $\alpha <1/2$.
Suppose that the function $f(s)$ is analytic in a region containing the disc $|s-s_0|\le r$. Suppose further that there is a number $A_1$ independent of $s$ such that
$$\left|\frac{f(s)}{f(s_0)}\right| \le A_1$$ 
for $|s-s_0|\le r$. Then, for any $s$ in the disc $|s-s_0|\le \alpha r$
we have
\begin{equation*}
\left| \frac{f'(s)}{f(s)} - \sum_\rho \frac{1}{s-\rho}\right| \le \frac{4\log A_1}{r(1 -2\alpha)^2},
\end{equation*}
where $\rho$ runs through the zeros of $f(s)$ in the disc
$|s -s_0|\le \tfrac{1}{2} r$, counted with multiplicity.
\end{lemma}

\begin{lemma}\cite[Lemma~6]{hiaryleongyangArxiv}\label{log_deriv_zeta_lem2}
Let $s$ and $s_0$ be complex numbers with real parts $\sigma$ and $\sigma_0$, respectively.
Let $r$, $\alpha$, $\beta$, $A_1$ and $A_2$ be positive numbers, 
possibly depending on $s_0$, such that $\alpha<1/2$ and $\beta<1$.
Suppose that the function $f(s)$ satisfies 
the conditions of Lemma \ref{log_deriv_zeta_lem1} with $r$, $\alpha$ and $A_1$, 
and that
$$\left| \frac{f'(s_0)}{f(s_0)}\right| \le A_2.$$ 
Suppose, in addition, that $f(s)
\neq 0$ for any $s$ in both the disc $|s-s_0|\le r$ and the right half-plane $\sigma \ge \sigma_0 - \alpha r$. 
Then, for any $s$ in the disc $|s-s_0|\le \alpha\beta r$,
\begin{equation*}
\left| \frac{f'(s)}{f(s)}\right| \le \frac{8\beta\log A_1}{r(1-\beta)(1-2\alpha)^2} + \frac{1+\beta}{1-\beta}A_2.
\end{equation*}
\end{lemma}

\begin{theorem}\label{andrews bound}
    Let $k\geq 3$ be an integer and $\sigma_k := 1-k/(2^k-2)$. Then for all $t\geq 3$, $$|\zeta(\sigma_k+it)|\leq 1.546t^{1/(2^k-2)}\log t.$$ 
\end{theorem}
\begin{proof}
Yang \cite[Theorem~1.1]{yang2023explicit} proves the result for all integers $k\ge 4$, and \cite[Theorem~1.1]{hiarypatelyang2022} implies the result for $k=3$, from $|\zeta(1/2+it)|\le 0.618t^{1/6}\log t$ for all $t\ge 3$.
\end{proof}

\begin{lemma}\label{zeta_log-powers}
Let $s = \sigma+it$ for real $\sigma, t$. Choose real parameters $\omega_1>0$ and $t_0 \ge 3$ such that
$$\omega_1 \ge \max_{t\geq t_0}\left(\frac{8\log\log t}{\log t}\right).$$
Also define
$$\omega_2 := \begin{cases} \frac{1}{\omega_1\log 2}\left(1-\frac{1}{e} + \frac{\log \omega_1}{\log\log t_0}\right) \quad &\text{if} \quad \omega_1 < 1 \\ \frac{1}{\omega_1\log 2}\left(1-\frac{1}{e} \right) \quad &\text{if} \quad \omega_1 \geq 1. \end{cases}$$
Then for every $\kappa>0$, all $t\ge t_0$, and $1-\frac{\omega_2(\log\log t)^2}{\log t} \leq \sigma\leq 1+\kappa$ we have
\begin{equation*}
|\zeta(s)| \le  A_{\kappa}\log^B t,
\end{equation*}
where $B=1+8/3\omega_1$ and
\begin{equation}\label{Dk}
    A_{\kappa} := 1.546\zeta(1+\kappa)\left(1+\frac{2+\kappa}{t_0} \right)^{1/6}\left(1 + \frac{2+\kappa}{t_0\log t_0}\right)\left(1+\frac{2\sqrt{1+\kappa}}{t_0} \right).
\end{equation}
We give some admissible values of $A_{\kappa},\, B,\, \kappa,\, \omega_1,\,\omega_2,$ and $t_0$ in Table \ref{tab:zeta-log-powers}.
\end{lemma}

\begin{proof}
We use the bound for $\zeta(s)$ given in Theorem \ref{andrews bound}. The Phragm{\'e}n--Lindel{\"o}f Principle can be used to show that a bound on $\zeta(s)$ of the form in Theorem \ref{andrews bound} holds for $\sigma_k\leq \sigma\leq 1+\kappa$ for any real $\kappa >0$ and integer $k\geq 3$. We refer to the statement of the principle in \cite[Lemma~3]{Trudgian_14} (or \cite[Lemma~2.1]{trudgian2016improvements2}), in which we consider $f(s) = (s-1)\zeta(s)$. As the lower bound on $k$ implies $\sigma_k\in[1/2,1]$, we can take $Q=1$ to satisfy the condition in \cite[(4.1)]{Trudgian_14} with $a=\sigma_k$, $b=1+\kappa$, $\alpha_1 = 1+1/(2^k-2)$, $\alpha_2=1$, $\beta_1=1$, $\beta_2=0$, $A=1.546$, and $B=\zeta(1+\kappa)$. This choice of $Q$ comes from separately checking $|t|\leq 3$ and $|t|\geq 3$. First, we can numerically verify
\begin{align*}
    |(s-1)\zeta(s)| \leq \begin{cases} 1.546|s+1|^{1+1/(2^k-2)}\log|s+1| \quad &\text{for}\ \sigma = \sigma_k,  \\ \zeta(1+\kappa)|s+1| \quad &\text{for}\ \sigma = 1+\kappa,  \end{cases}
\end{align*}
both for $|t|\leq 3$, noting that the bound for $\sigma=\sigma_k$ follows from
\begin{equation*}
    \max_{\substack{1/2\leq \sigma<1 \\ |t|\leq 3}}\Big( |(s-1)\zeta(s)| - 1.546|s+1|\log|s+1| \Big)< -0.21,
\end{equation*}
the maximum occurring at $t=0$ and $\sigma=1/2$. Second, it can be seen that
\begin{align*}
    |s-1|t^{1/(2^k-2)}\log t &\leq |s+1|^{1+1/(2^k-2)}\log|s+1|
\end{align*}
for $\sigma=\sigma_k$ and $|t|\geq 3$, and lastly $$|(s-1)\zeta(s)| \leq \zeta(1+\kappa)|s+1|$$ for $\sigma=1+\kappa$ and $|t|\geq 3$.
Then, by \cite[Lemma~3]{Trudgian_14} we have
\begin{align}\label{zeta-k-bound}
    |\zeta(s)| &\le 1.546\zeta(1+\kappa)\frac{|s+1|^{1+1/(2^k-2)}\log|s+1|}{|s-1|} \nonumber \\
    &\leq A_{\kappa} t^{1/(2^k-2)}\log t
\end{align}
for $\sigma_k\leq \sigma\leq 1+\kappa$ and $t\geq t_0\geq 3$, where $A_\kappa$ is defined in \eqref{Dk}.

From here we can use the proof of Titchmarsh's Theorem 5.17 \cite{titchmarsh1939theory}. Let $$k=\left\lfloor \frac{1}{\log 2}\log\left( \frac{\omega_1\log t}{\log\log t}\right) \right\rfloor,$$ where $\omega_1$ is chosen such that we have $k \ge 3$ for all $t\geq t_0$. This $k$ implies
\begin{align}\label{sk-bound}
    \frac{k}{2^k-2} \geq \frac{k}{2^k} \geq \frac{1}{\omega_1 \log 2}\log\left( \frac{\omega_1\log t}{\log\log t}\right) \frac{\log\log t}{\log t} \ge \frac{\omega_2(\log\log t)^2}{\log t}
\end{align}
with $\omega_2$ as defined in the statement of the lemma.

The bound in (\ref{sk-bound}) indicates that since (\ref{zeta-k-bound}) holds for $\sigma\geq \sigma_k$, it also holds for $$\sigma\geq 1-\frac{\omega_2(\log\log t)^2}{\log t}.$$ The expression for $k$ also implies 
\begin{align*}
    \frac{1}{2^k-2} \leq \frac{4}{3\cdot2^k} \leq \frac{8\log\log t}{3\omega_1\log t},
\end{align*}
so (\ref{zeta-k-bound}) becomes
\begin{align*}
    |\zeta(s)| \leq A_{\kappa} t^{1/(2^k-2)}\log t\leq A_{\kappa} t^{8\log\log t/(3\omega_1\log t)}\log t = A_{\kappa} (\log t)^{1+8/3\omega_1},
\end{align*} 
where $A_{\kappa}$ is in \eqref{Dk}.
For specific $t_0$, we can find a $\kappa$ that minimises $A_{\kappa}$, then find the smallest $\omega_1$ that permits our choice of $k$ for all $t\geq t_0$. It is then possible to compute $B$ and $\omega_2$. This allows us to compute the values in Table \ref{tab:zeta-log-powers} mentioned in the statement of the lemma.
\end{proof}

We can compare this result to that of Hiary, Leong and Yang \cite[Lemma~7]{hiaryleongyangArxiv}, derived using a Richert-type bound for $\zeta(s)$ by Ford in \cite{ford_2002_zeta-bounds}. For instance, in the range $t\geq e^e$ and $1-\frac{0.309(\log\log t)^2}{\log t}\leq \sigma\leq 2$, we have $$|\zeta(s)|\leq 2.5(\log t)^{1.91},$$ whereas \cite[Lemma~7]{hiaryleongyangArxiv} has $$|\zeta(s)|\leq 76.2(\log t)^{3.29}.$$

\begin{table}[H]
    \centering
    \begin{tabular}{|c|c|c|c|c|c|}\hline
        $t_0$ & $\kappa$ & $A_{\kappa}$ & $\omega_1$ & $B$     & $\omega_2$  \\ \hline
        $3$ & $1.5$    & $10$        & $8/e$   & $1.91$ & $0.309$ \\
        $e^e$ & $3.2$    & $2.5$        & $8/e$   & $1.91$ & $0.309$ \\
        $500$ & $8.1$    & $1.6$        & $2.36$ & $2.13$ & $0.386$ \\
        $H$   & $41.5$   & $1.6$        & $0.94$   & $3.84$  & $0.941$ \\ \hline
    \end{tabular}
    \caption{Admissible values of parameters in Lemma \ref{zeta_log-powers}.}
    \label{tab:zeta-log-powers}
\end{table}

\begin{lemma}\cite[Corollary~1.2]{yang2023explicit}\label{littlewood_zero_free_reg}
There are no zeroes of $\zeta(\sigma+it)$ in the region
\begin{equation*}
\sigma>1- \frac{\log\log t}{21.233\log t}, \qquad t\ge 3.
\end{equation*}
\end{lemma}

\begin{lemma}\cite[Lemma~5.4]{ramare2016explicit}\label{bastien bound}
Let $\gamma$ denote Euler's constant. For $\sigma >1$ we have
\begin{equation*}
|\zeta(\sigma + it)| \le \zeta(\sigma) \le \frac{e^{\gamma(\sigma - 1)}}{\sigma - 1}.
\end{equation*}
\end{lemma}

\begin{theorem}\cite[Theorem~1]{hiaryleongyangArxiv}\label{hiary 1 bound}
For $t\geq 3$ we have
\begin{align*}
|\zeta(1+it)| &\leq 1.731 \frac{\log t}{\log\log t}.
\end{align*}
\end{theorem}

\begin{lemma}\label{fragment of lindel}
Let $a,b,Q,$ and $k$ be real numbers such that $Q+a>e$, and let $f(s)$ be a holomorphic function in the strip $a\le \textup{Re}(s) \le b$ such that it satisfies the growth condition
\begin{equation*}
|f(s)|< C\exp\big(e^{k|t|}\big)
\end{equation*}
for some $C>0$ and $0<k<\pi/(b-a)$. Suppose further that 
    \begin{align}\label{P-L-condtions}
        |f(s)|\leq \begin{cases}
            A|Q+s|^{\alpha_1}\left(\frac{\log|Q+s|}{\log\log|Q+s|}\right)^{\alpha_2} \quad \text{for $\Re(s) = a$} \\
            B|Q+s|^{\beta_1}\left(\frac{\log|Q+s|}{\log\log|Q+s|}\right)^{\beta_2} \quad \text{for $\Re(s) = b$}, 
        \end{cases}
    \end{align}
where $\alpha_1 \ge \beta_1$ and where $A,B,\alpha_1,\alpha_2, \beta_1, \beta_2 \ge 0$. Then throughout the strip $a\le \textup{Re}(s) \le b$, one has
\begin{align*}
|f(s)|\le &\left(A|Q+s|^{\alpha_1}\left(\frac{\log|Q+s|}{\log\log|Q+s|}\right)^{\alpha_2} \right)^{\tfrac{b-\textup{Re}(s)}{b-a}}\\
&\qquad\qquad\boldsymbol{\cdot}\left(B|Q+s|^{\beta_1}\left(\frac{\log|Q+s|}{\log\log|Q+s|}\right)^{\beta_2}\right)^{\tfrac{\textup{Re}(s)-a}{b-a}}.
\end{align*}
\end{lemma}
\begin{proof}
This is derived in a similar way to Lemma 3 of \cite{Trudgian_14}. Trudgian's lemma is an adaption of a theorem from Rademacher on a generalisation of the Phragm{\'e}n--Lindel{\"o}f theorem \cite[\S33]{rademacher1973Topics}. The Phragm{\'e}n--Lindel{\"o}f theorem \cite[\S29]{rademacher1973Topics} is applied to a function $F(s)$, to which Trudgian adds a factor of $\log(Q+s)$, to use the theorem for bounds on $f(s)$ that have a factor of $\log|Q+s|$.

We will make a similar change to Rademacher's theorem. However, instead of a factor of $\log(Q+s)$, we will include a factor of $\log(Q+s)/\log\log(Q+s)$ to form the function
\begin{equation*}
F(s) = f(s)\phi(s;Q)E^{-1}e^{-vs}\left( \frac{\log (Q+s)}{\log\log(Q+s)}\right)^{\frac{\alpha_2(\sigma-b)+\beta_2(a-\sigma)}{b-a}},
\end{equation*}
where $E$ is defined in \cite[Lem.~2]{Trudgian_14}.
We also include the restriction $Q+a>e$ to ensure that $F(s)$ is holomorphic in the strip $a\le \sigma\le b$. The proof then proceeds as in \cite[Lemma~3]{Trudgian_14} and \cite[\S32]{rademacher1973Topics}.
\end{proof}

\begin{lemma}\label{zeta_bound_after_1}
    For $1\leq \sigma\leq \sigma_1$ with $\sigma_1>1$, and $t\geq t_0\geq 3$, we have
    \begin{equation*}
        |\zeta(\sigma+it)| \le Z(\sigma_1, t_0) \frac{\log t}{\log\log t},
    \end{equation*}
    where $$ Z(\sigma_1, t_0) := 1.731\zeta(\sigma_1)\left(1+\frac{3}{t_0}\right)\frac{\log(t_0+\sigma_1+2)}{\log t_0}.$$ Admissible values of $Z$ are given in Table \ref{tab:fragmentz} for specific $t_0$ and optimized $\sigma_1$.
\end{lemma}
    \begin{table}[h]
        \centering
        \begin{tabular}{|c|c|c|}\hline
            $t_0$ & $\sigma_1$ & $Z$ \\ \hline
            $3$ & $4.32$ & $7.479$ \\
            $e^e$ & $5.88$ & $2.439$ \\
            $2e^e$ & $6.77$ & $2.065$ \\
            $500$ & $11.14$ & $1.750$ \\
            $1000$ & $11.97$ & $1.741$ \\ 
            $H$ & $11.96$ & $1.732$ \\ \hline
        \end{tabular}
        \caption{Values for $Z$ in Lemma \ref{zeta_bound_after_1} for specific $t_0$ after optimizing over $\sigma_1$.}
        \label{tab:fragmentz}
    \end{table}
\begin{proof}
    This is similar to the proof of Lemma \ref{zeta_log-powers}. We take $f(s) := (s-1)\zeta(s)$ in Lemma \ref{fragment of lindel} and bound $\zeta(s)$ in $1\leq \sigma\leq \sigma_1$ for $\sigma_1 >1$. Note that $f(s)$ is entire in this strip, and satisfies the growth condition given by Rademacher or \cite[Lemma~3]{Trudgian_14}. 
    
    We will need two bounds on $|\zeta(s)|$ to verify the condition in \eqref{P-L-condtions}: one for $\sigma = 1$ and one for $\sigma=\sigma_1$. The latter is relatively simple to derive, as we can use $|\zeta(\sigma_1+it)| \leq \zeta(\sigma_1)$, which is true for all $t$ and computable. For the other bound we use Theorem \ref{hiary 1 bound}.
    
    With $a=1$, $b=2$, $\alpha_1 = 1$, $\alpha_2 = 1$, $\beta_1 = 1$, $\beta_2 = 0$, $A=1.731$ and $B=\zeta(\sigma_1)$, we will take $Q=2$ to check that \eqref{P-L-condtions} holds for both $|t|\leq 3$ and $|t|\geq 3$. First, it can be numerically verified that for $\sigma = 1$,
    \begin{align*}
        |(s-1)\zeta(s)| \leq 1.731|s+2|\frac{\log |s+2|}{\log\log |s+2|}
    \end{align*}
    for $|t|\leq 3$. Note that the choice of $Q$ comes from this particular condition. Second, Theorem \ref{hiary 1 bound} implies that
    \begin{align*}
        |(s-1)\zeta(s)| \leq 1.731\frac{t\log t}{\log\log t} &\leq 1.731|s+2|\frac{\log |s+2|}{\log\log |s+2|}
    \end{align*}
    for $\sigma=1$ and $|t|\geq 3$. Lastly, $$|(s-1)\zeta(s)| \leq \zeta(\sigma_1)|s+2|$$ holds for $\sigma=\sigma_1$ and all $t$.
    
    Then, following a similar conclusion to \cite[Lemma~3]{Trudgian_14}, we have
    \begin{align*}
        |\zeta(s)| &\le 1.731\zeta(\sigma_1)\frac{|s+2|}{|s-1|} \frac{\log |s+2|}{\log\log |s+2|} \leq Z(\sigma_1, t_0) \frac{\log t}{\log\log t},
    \end{align*}
     for $1\leq\sigma\leq \sigma_1$ and $t\geq t_0\geq 3$. 
     Finally, we optimize over $\sigma_1$ with the aim of minimizing $Z$, to obtain the values in Table \ref{tab:fragmentz}, as mentioned in the statement of the lemma. 
\end{proof}

\begin{lemma}\label{inverse zeta trig poly}
Let $d, \, \sigma_1$ be real positive parameters, $\sigma = 1+ d\log\log t/\log t$, $\gamma$ be Euler's constant, and $Z(\sigma_1,t_0)$ be defined as in Lemma \ref{zeta_bound_after_1}. Then for $t\ge t_0\ge 3$ and all $\sigma >1$, we have
\begin{equation*}
    \left| \frac{1}{\zeta(\sigma +it)}\right| \le V_1(d,\sigma_1,t_0)\frac{\log t}{\log\log t},
\end{equation*}
where
\begin{equation}\label{V1}
V_1(d,\sigma_1,t_0) := \left(\frac{1}{d} +\frac{\log\log t_0}{\log t_0} \right)^{3/4}\left( Z(\sigma_1,2t_0)\left( 1 + \frac{\log 2}{\log t_0} \right)\right)^{1/4}.
\end{equation}
In addition, for $\sigma >1$ such that $\sigma \le \sigma_1$ and $\sigma \le (1+\gamma)/\gamma$, we have for $t\ge t_0\ge 3$ that
\begin{equation*}
    \left| \frac{1}{\zeta(\sigma +it)}\right| \le V_2(d,\sigma_1,t_0)\frac{\log t}{\log\log t},
\end{equation*}
where
\begin{equation}\label{V2}
V_2(d,\sigma_1,t_0) := \left(\frac{1}{d}\exp\left(\frac{\gamma d\log\log t_0}{\log t_0}\right) \right)^{3/4}  \left( Z(\sigma_1,2t_0) \left( 1 + \frac{\log 2}{\log t_0} \right)\right)^{1/4}.
\end{equation}
\end{lemma}
\begin{proof}
Since $\sigma >1$, we use the classical non-negativity argument involving the trigonometric polynomial $2(1+\cos\theta)^2=3 + 4\cos\theta + \cos2\theta$ (see \cite[\S~3.3]{titchmarsh1986theory}). That is, we use $\zeta^3(\sigma)|\zeta^4(\sigma+it)\zeta(\sigma +2it)| \ge 1$, which implies
\begin{equation}\label{trig ineq}
    \left| \frac{1}{\zeta(\sigma+it)}\right| \le |\zeta(\sigma)|^{3/4}|\zeta(\sigma + 2it)|^{1/4}.
\end{equation}
Taking $\sigma = 1+\delta := 1+ d\log\log t/\log t$ in the above gives
\begin{equation}\label{trig poly ineq}
    \left| \frac{1}{\zeta(1 + \delta +it)}\right| \le |\zeta(1 + \delta)|^{3/4}|\zeta(1 + \delta + 2it)|^{1/4}.
\end{equation}
These two factors can be bounded with Lemmas \ref{bastien bound} and \ref{zeta_bound_after_1}. First consider $1<\sigma\leq \min\{\sigma_1, (1+\gamma)/\gamma\}$. In this range, $e^{\gamma\delta}/\delta$ is decreasing. Hence, we have
\begin{align*}
    \frac{e^{\gamma\delta}}{\delta} &\le \frac{1}{d}\exp\left(\frac{\gamma d\log\log t_0}{\log t_0}\right) \frac{\log t}{\log\log t}, \\
    \frac{\log 2t}{\log\log 2t} &\le \left( 1 + \frac{\log 2}{\log t_0} \right) \frac{\log t}{\log\log t}. 
\end{align*}
Applying these bounds in Lemmas \ref{bastien bound} and \ref{zeta_bound_after_1}, and then \eqref{trig poly ineq}, we find
\begin{equation*}
\left| \frac{1}{\zeta(1 + \delta +it)}\right| \le V_2(d,\sigma_1,t_0) \frac{\log t}{\log\log t}.
\end{equation*}

Second, to obtain a bound over all $\sigma >1$, we simply repeat the process with the trivial bound $\zeta(\sigma)\le \sigma/(\sigma-1)$ in place of Lemma \ref{bastien bound}, as it is decreasing for all $\sigma>1$. However, it is worth noting that this bound is not as sharp as Lemma \ref{bastien bound} for $\sigma$ close to $1$.
\end{proof}


\section{Bounding $|\zeta'(s)/\zeta(s)|$}\label{mainsec1}

In this section, we use the lemmas in the previous section to bound $|\zeta'(s)/\zeta(s)|$. We follow the argument of \cite{trudgian2015explicit}, which gives an explicit version of results in \cite[\S3]{titchmarsh1986theory}. This has been done in \cite{hiaryleongyangArxiv} to obtain bounds for $|\zeta'(s)/\zeta(s)|$ and $|1/\zeta(s)|$ on the $1$-line, and we now extend these results to hold within a zero-free region. The method is identical to that of \cite{hiaryleongyangArxiv}, but uses sharper bounds on $\zeta$, and a more careful handling of constants to obtain some savings. Given this, we will only give an outline of how to obtain such bounds, defining the functions with the same notation, but omitting details. The reader is referred to \cite{hiaryleongyangArxiv} for a full elucidation on the proof, especially on justification of certain parameter choices.

As before, we construct concentric discs, centred just to the right of the line $\sigma = 1+it$, and extend them slightly to the left and into the critical strip. We aim to apply Lemmas \ref{log_deriv_zeta_lem1} and \ref{log_deriv_zeta_lem2} with $f(s)=\zeta(s)$. 

Let $t_0\ge e^e$ and $t'\ge t_0$ be constants. The centre of the concentric discs will be denoted by 
\begin{equation*}
s' = \sigma' + it' = 1+\delta_{t'} +it' = 1+\frac{d \log\log t' }{ \log t'} +it',
\end{equation*}
where $d$ is a real positive constant to be chosen later. Notice that $\delta_{t'}$ is decreasing for $t'>e^e$, and is at most $d/e$. Also, let $r$ and $\epsilon\le 1$ be positive parameters chosen such that $\sigma' +r \le 1+\epsilon$, which will be satisfied if
\begin{equation}\label{r_eps_cond}
    \frac{d}{e}+r\le \epsilon.
\end{equation}

Let $s=\sigma+it$ be a complex number. 
Aiming to apply Lemma~\ref{log_deriv_zeta_lem1} in the disc $|s-s'|\le r$,
we first seek a valid $A_1$ in that disc. 
Similarly, for Lemma~\ref{log_deriv_zeta_lem2}, we seek a valid $A_2$ at the disc centre $s=s'$. In applying Lemma~\ref{log_deriv_zeta_lem2}, we also need to ensure that the non-vanishing condition on $f(s)=\zeta(s)$ is met, which we do with a zero-free region.

First, let 
\begin{equation}\label{C_10}
\sigma_{1,t'} :=  1-\frac{C_1\,(\log\log t')^2}{\log t'}, \,\,
C_1 := \omega_2 \cdot \min \left\lbrace \left(\frac{\log\log (t'-\epsilon)}{\log\log t'}\right)^2, 
\frac{\log t'}{\log (t'+  \epsilon)}\right\rbrace,
\end{equation}
where $\omega_2$ is defined in Lemma \ref{zeta_log-powers}. Then, by the same lemma, for each $t\in [t'-r,t'+r]$ and any $\sigma \in [\sigma_{1,t'},1+\kappa]$, we have
\begin{equation}\label{A3_lb0}
|\zeta(s)| \le A_\kappa\log^B (t'+r) \leq A_3\log^B t', \quad A_3 = A_3(t') := A_{\kappa}\left(1+\frac{\log(1+\epsilon/t')}{\log t'}\right)^B,
\end{equation}
with $A_{\kappa}$ and $B$ given in Table \ref{tab:zeta-log-powers}.

For our purposes, we choose a radius of
\begin{equation}\label{r_0.5_cond0}
r=r_{t'}:= \left(C_1+\frac{d}{\log\log t'}\right)\frac{(\log\log t')^2}{\log t'}\le \frac{C_1(\log\log t')^2}{\log t'} + \delta_{t'}.
\end{equation}
To ensure the constraint \eqref{r_eps_cond} is met for all $t'\ge t_0$, we need 
\begin{equation}\label{r_eps_cond0}
 \frac{d}{e}+\left(C_1+d\right)\frac{4}{e^2}\le \epsilon \le 1,
\end{equation}
as $(\log \log t')^2/\log t'$ reaches a maximum of $4/e^2$ 
at $t'=t_e:=e^{e^2}$. This choice of $r$ also requires Lemma \ref{zeta_log-powers} to be applicable for $\sigma$ up to $1+\delta_{t'}+r$, which corresponds to
\begin{equation*}
\kappa \ge \delta_{t'}+r = \left(C_1+\frac{2d}{\log\log t'}\right)\frac{(\log\log t')^2}{\log t'}.
\end{equation*}
This will be satisfied for all $t_0\geq e^e$ for $$\kappa \geq \frac{4}{e^2}\left( \omega_2 + \frac{2d}{\log\log t_0} \right),$$ and the values of $\kappa$ in Table \ref{tab:zeta-log-powers} satisfy this condition.

To fulfil the bounding condition of Lemma \ref{log_deriv_zeta_lem1}, we can take either lower bound for $|\zeta(s')|$ from Lemma \ref{inverse zeta trig poly}, and use a constant
\begin{equation}\label{V min}
    V \in \left\lbrace V_1(d,\sigma_1,t_0), \, V_2(d,\sigma_1,t_0) \right\rbrace,
\end{equation}
where $d$ and $t'$ are always taken such that $1+\delta_{t'}\leq (1+\gamma)/\gamma$. With this, by \eqref{A3_lb0} and \eqref{r_0.5_cond0}, throughout the disc $|s-s'|\le r$ we have
\begin{align*}
\left| \frac{\zeta(s)}{\zeta(s')}\right| \le &\, A_3\,V\frac{\log^{B+1}t'}{\log\log t'}.
\end{align*}
Noting that $A_3$ is decreasing in $t'\ge t_0$, we obtain that throughout $|s-s'|\le r$,
\begin{equation}\label{AX max}
   \log \left| \frac{\zeta(s)}{\zeta(s')}\right| \le \log A_1 :=
   (B+1)\log\log t' -\log\log\log t' + \log\left( A_{\max}\,V\right),
\end{equation}
where $A_{\max} = A_3(t_0)$.

Next, we fulfill the zero-free condition of Lemma \ref{log_deriv_zeta_lem2} by way of the zero-free region in Lemma \ref{littlewood_zero_free_reg}, where for $t\ge 3$,
\begin{equation}\label{zero-free}
\zeta(\sigma+ it)\neq 0\quad\text{for}\quad \sigma> 1-\frac{c_0\log\log t}{\log t}, \quad \text{where}\quad 
c_0:=\frac{1}{21.233}.
\end{equation}
This leads to the following choice of
\begin{equation}\label{alphar}
 \alpha = \frac{1}{r}\cdot \frac{(d+c_1)\log\log t'}{\log t'} = 
 \frac{d+c_1}{d+C_1 \log\log t'},
\end{equation}
where
\begin{equation}\label{c0 bound}
    c_1 := \frac{c_0 \log t_0}{\log(t_0+\epsilon)},
\end{equation}
to keep $c_1 \le c_0$, and we require
\begin{equation}\label{alpha cond}
    \frac{d+c_1}{d+C_1 \log\log t_0}<\frac{1}{2}
\end{equation}
to satisfy the constraint $\alpha<1/2$ in Lemma \ref{log_deriv_zeta_lem1}.

To determine $A_2$, we utilise \cite[Lemma 70.1]{tenenbaum1988rrhall}, 
which states that for $\sigma >1$,
\begin{equation}\label{log-deriv real-part bound}
\left| \frac{\zeta'(s)}{\zeta(s)}\right| \le -\frac{\zeta'(\sigma)}{\zeta(\sigma)} < \frac{1}{\sigma -1}.
\end{equation}
Taking $s=s'$ in this inequality leads to
\begin{equation*}
\left| \frac{\zeta'(s')}{\zeta(s')}\right| < \frac{1}{\delta_{t'}} = A_2 := \frac{\log t'}{d \log\log t'}.
\end{equation*}

Finally, we turn to the choice of $\beta$ in Lemma \ref{log_deriv_zeta_lem2}. From here on, we differ from \cite{hiaryleongyangArxiv} as we want the circles to be able to extend into the zero-free region. For any results from Lemma \ref{log_deriv_zeta_lem2} to be meaningful to us, we want the left-most point of $|s-s'|\le \alpha\beta r$ to lie between the zero-free region and not exceed the $1$-line. In other words, we want $\sigma'-\alpha\beta r \le 1$, which means $\beta$ must satisfy
\begin{equation}\label{beta cond}
1>\beta \ge \frac{\delta_{t'}}{\alpha r} = \frac{d}{c_1+d}.
\end{equation}
Up to this point, Lemma \ref{log_deriv_zeta_lem2} holds for
\begin{equation*}
1- (\beta c_1 - d(1-\beta))\frac{\log\log t'}{\log t'} \le \sigma' \le 1+ (\beta c_1 + d(1+\beta))\frac{\log\log t'}{\log t'}.
\end{equation*}
For larger $\sigma'$, we will just use \eqref{log-deriv real-part bound}. To finish, observe that $1/(1-2\alpha)^2$ is decreasing in $t'$, so we can substitute $t'\mapsto t$ to arrive at the following lemma.

\begin{lemma}\label{zeta'/zeta_main_lem}
Let $t_0\ge e^e$, $\sigma_1 \ge \sigma$, $d>0$, and $\epsilon\le 1$ be constants satisfying the constraints \eqref{r_eps_cond0} and \eqref{alpha cond}. Let $\beta >0$ (subject to \eqref{beta cond}) and $c_1=c_1(\epsilon,t_0)$ (defined in \eqref{c0 bound}) be chosen such that $W := 1/(\beta c_1 -d(1-\beta))>1/c_0$, where $c_0:=1/21.233$. Then, for $t\ge t_0$ and
$$\sigma \ge 1- \frac{\log\log t}{W\log t}$$
we have
\begin{equation*}
\left| \frac{\zeta'(\sigma +it)}{\zeta(\sigma+it)}\right| \le Q_1(d,\epsilon,\sigma_1,t_0)\frac{\log t}{\log\log t},
\end{equation*}
where 
\begin{align*}
Q_1(d,\epsilon,\sigma_1,t_0) := \min_{V}\Bigg\lbrace\max &\Bigg\lbrace \frac{1}{\beta c_1 + d(1+\beta)} ,\, \\
&\lambda_1 \left( B+1 + \frac{\log\left( A_{\max}\,V\right)}{\log\log t_0}\right)
+ \lambda_2\Bigg\rbrace \Bigg\rbrace,
\end{align*}
with
\begin{equation*}
\begin{split}
    \lambda_1 := \frac{8\beta}{C_1(1-\beta)} \left(\frac{C_1\log\log t_0+d}{C_1\log \log t_0-d-2c_1} \right)^2, \qquad
    \lambda_2 :=\frac{1+\beta}{d(1-\beta)}.
\end{split}
\end{equation*}
The constant $B$ is defined in Lemma~\ref{zeta_log-powers}, $A_{\max}$ is defined in \eqref{AX max} and depends on $t_0$, $C_1$ is defined in \eqref{C_10} and depends on $\epsilon$ and $t_0$, and $V$ is defined in \eqref{V min} and depends on $d,\,\sigma_1$ and $t_0$. 
\end{lemma}

\subsection{Computations}

We can compute admissible $R_1$ in Corollary \ref{cor:zeta'/zeta_main_lem} using Lemma \ref{zeta'/zeta_main_lem}.

\begin{proof}[Proof of Corollary \ref{cor:zeta'/zeta_main_lem}]
     Although we can use Lemma \ref{zeta'/zeta_main_lem} with $t_0=e^e$, we can reduce the size of $Q_1$ by computing it for $t_0 = H$, the height to which the Riemann Hypothesis has been verified in \cite{PlattTrudgianRH}, and combining it with a result for $e^e\le t_0 \leq t<H$. Over the latter range, we know there are no zeros of $\zeta(s)$ with $\sigma> 1/2$. We can therefore use a wider zero-free region than in Lemma \ref{littlewood_zero_free_reg} when verifying the condition in Lemma \ref{log_deriv_zeta_lem2}. Since
     \begin{equation*}
         \frac{1}{2} < 1- \frac{(e/2)\log\log t}{\log t}
     \end{equation*}
      for all $t>0$, we can replace $c_0$ with any $c_{RH} < e/2$ in \eqref{zero-free} as long as it fulfils \eqref{alpha cond}. Given \eqref{c0 bound}, we see that \eqref{alpha cond} is satisfied by 
     \begin{equation}\label{crh}
         c_{RH} < \frac{C_1\log\log t_0 -d}{2},
     \end{equation}
    since $\log(t_0+\epsilon)/\log t_0 >1$ for all $t_0>0$. When optimising over $d$ to calculate $Q_1$ for $t\ge H$, we consistently found an optimal $d< 0.05$. Hence, we will assume $d< 0.05$, which, along with $\epsilon\leq 1$ and $t_0 = e^e$, gives us that the right-hand side of \eqref{crh} is 
     \begin{equation*}
         \frac{1}{2}\left(0.309\cdot\min \left\lbrace (\log\log (e^e-1))^2, \frac{e}{\log (e^e+ 1)}\right\rbrace -0.05\right)\ge 0.121\ldots,
     \end{equation*}
    by Table \ref{tab:zeta-log-powers}. Therefore, we can choose any $c_{RH} \le 1/8.22 < 0.121\ldots < e/2$ for the range of values $e^e\leq t\leq H$.

    Although a larger $c_{RH}$ would mean we have a larger zero-free region, taking it too large can have an adverse effect on the parameter $\alpha$. Smaller $c_{RH}$ reduces $\alpha$, which means a better bound in Lemma \ref{log_deriv_zeta_lem2}, and hence in Lemma \ref{zeta'/zeta_main_lem}. However, larger $c_{RH}$ means we can take smaller $\beta$ to satisfy the expression for $W$, which leads to a smaller value for $\lambda_1$ in Lemma \ref{zeta'/zeta_main_lem}.

    Starting by fixing a value for $W$, we compute $Q_1$ in the range $t\ge H$ by optimising over $\epsilon$, $\beta$, and $d$, subject to the constraints of Lemma \ref{zeta'/zeta_main_lem}, and using $c_0$ as in \eqref{zero-free}. We similarly obtain $Q_1$ in the range $t_0\leq t\leq H$, but instead optimise over $\epsilon$, $\beta$, $d<0.05$, and $c_{RH} \le 1/8.22$. We also require the result to hold for
    \begin{equation*}
        \sigma' -\alpha \beta r \le 1- \frac{\log\log t}{W\log t},
    \end{equation*}
    which leads to replacing the condition in \eqref{beta cond} with
    \begin{equation*}
    \beta \ge \frac{d+ 1/W}{d+ c_1}.
    \end{equation*}
    Finally we take the maximum of the two results, which is an admissible $Q_1$ over all $t\geq t_0$. This value is $R_1$ in Corollary \ref{cor:zeta'/zeta_main_lem}. In Table \ref{tab:log-deriv-zeta}, we opt to use the smallest $t_0$ that makes the $Q_1$ over $t_0\leq t\leq H$ the smallest of the two $Q_1$. For $c_{RH}$ we use $12$ for all entries up to and including $W=31$; for $W=35$ we use $c_{RH}=10.5$, and for all larger $W$ we use $c_{RH}=9$.
   \end{proof}

\begin{table}[h!]
\centering
\setlength\tabcolsep{10pt}
\footnotesize
\begin{tabular}{|c|c|c||c|c|c|c|c|} 
 \hline
 $W$ & $t_0$ & $R_1$ & $\alpha_1$ & $\epsilon$ & $d$ & $\beta$ & $\sigma_1$ \\ \hline
 $21.24$ & $e^e$ & $586798$ & $0.02747$ & $0.984$ & $0.040833$ & $0.999823479$ & $9.59$ \\ 
 $21.3$ & $e^e$ & $61411$ & $0.02736$ & $0.669$ & $0.040469$ & $0.998308198$ & $13.53$ \\
 $21.4$ & $e^e$ & $24793$ & $0.02899$ & $0.571$ & $0.045826$ & $0.996044781$ & $13.02$\\
 $21.5$ & $e^e$ & $15547$ & $0.02615$ & $0.794$ & $0.036503$ & $0.993003844$ & $12.71$\\
 $21.6$ & $e^e$ & $11348$ & $0.02902$ & $0.75$ & $0.045935$ & $0.991398549$ & $9.02$\\
 $21.7$ & $e^e$ & $8918$ & $0.02822$ & $0.603$ & $0.043309$ & $0.988788889$ & $9.77$\\
 $21.8$ & $e^e$ & $7367$ & $0.02661$ & $0.792$ & $0.037998$ & $0.985604987$ & $13.36$\\
 $21.9$ & $e^e$ & $6272$ & $0.02742$ & $0.602$ & $0.040676$ & $0.983657702$ & $7.52$\\
 $22$ & $e^e$ & $5471$ & $0.02706$ & $0.623$ & $0.039479$ & $0.981034372$ & $6.45$\\
 $22.5$ & $e^e$ & $3357$ & $0.02772$ & $0.767$ & $0.041652$ & $0.970117219$ & $6.97$\\
 $23$ & $e^e$ & $2439$ & $0.02704$ & $0.598$ & $0.039411$ & $0.958174313$ & $9.93$\\ 
 $24$ & $18$ & $1599$ & $0.02731$ & $0.837$ & $0.040295$ & $0.93786755$ & $9.83$\\ 
 $25$ & $24$ & $1205$ & $0.0273$ & $0.908$ & $0.040274$ & $0.918777327$ & $13.41$\\
 $26$ & $33$ & $976$ & $0.02723$ & $0.735$ & $0.040044$ & $0.90090744$ & $13.38$\\ 
 $27$ & $45$ & $826$ & $0.02698$ & $0.583$ & $0.039199$ & $0.883429874$ & $8.34$\\ 
 $29$ & $85$ & $643$ & $0.02645$ & $0.562$ & $0.037456$ & $0.850817294$ & $7.11$\\ 
 $31$ & $163$ & $534$ & $0.02659$ & $0.849$ & $0.037945$ & $0.825515109$ & $8.04$\\ 
 $35$ & $300$ & $412$ & $0.0262$ & $0.583$ & $0.036661$ & $0.778825508$ & $10.65$\\
 $40$ & $490$ & $332$ & $0.02609$ & $0.956$ & $0.036301$ & $0.73504714$ & $10.04$\\
 $50$ & $500$ & $256$ & $0.0257$ & $0.607$ & $0.035005$ & $0.669962863$ & $14.27$\\
 $60$ & $500$ & $219$ & $0.02543$ & $0.838$ & $0.034113$ & $0.625294336$ & $11.66$\\ 
 $70$ & $500$ & $197$ & $0.02519$ & $0.81$ & $0.033352$ & $0.592150687$ & $14.77$\\
 \hline
\end{tabular}
\caption{Values for $W,R_1$ in Corollary \ref{cor:zeta'/zeta_main_lem}, based on $(\alpha, \beta, d, \epsilon, \sigma_1)$ from Lemma \ref{zeta'/zeta_main_lem}, where $\alpha_1$ is an upper bound on $\alpha$, and each entry is valid for $t\ge t_0$.}
\label{tab:log-deriv-zeta}
\end{table}

In addition, we have Corollary \ref{cor:zeta'/zeta_main_lem1}, especially relevant for $\sigma=1$.

\begin{proof}[Proof of Corollary \ref{cor:zeta'/zeta_main_lem1}]
    The process of obtaining the bound for $\sigma \geq 1$ is the same as that of \cite[Theorem~2]{hiaryleongyangArxiv}, but with the corresponding changes in $A_{\max}$, $B$, and $C_1$ due to Lemma \ref{zeta_log-powers}. Note that the method is essentially the same as our Corollary \ref{cor:zeta'/zeta_main_lem}, but we take $\beta = d/(c_1+d)$. The rest of the proof follows mutatis mutandis.

    As per Corollary \ref{cor:zeta'/zeta_main_lem}, the computations are split into cases $t_0\le t\le H$ and $t\ge H$, and we take the larger of the two resulting constants. An optimised value of $c_{RH}$ is used for each $t_0<H$, and we note that both entries in Table \ref{tab:log-deriv-zeta1} are determined by the value of $Q_1$ for $t_0\le t\le H$.
\end{proof}

\begin{table}[h!]
\centering
\setlength\tabcolsep{10pt}
\begin{tabular}{|c|c||c|c|c|c|c|} 
 \hline
 $t_0$ & $K_1$ & $1/c_{RH}$ & $\alpha_1$ & $d$ & $s_1$ & $\epsilon$ \\ \hline
 $e^e$ & $238.4$ & $16.7$ & $0.2188$ & $0.00946$ & $7.89$ & $0.1745$  \\
 $500$ & $113.3$ & $8.3$ & $0.1938$ & $0.0201$ & $11.26$ & $0.4098$  \\
 $H$ & $110.6$ & $21.233$ & $0.0240$ & $0.0295$ & $8.87$ & $0.8815$  \\
 \hline
\end{tabular}
\caption{Values for $K_1$ in Corollary \ref{cor:zeta'/zeta_main_lem1}, based on $(\alpha, d, \epsilon)$ from Lemma \ref{zeta'/zeta_main_lem}, where $\alpha_1$ is an upper bound on $\alpha$, and each entry is valid for $t\ge t_0$.}
\label{tab:log-deriv-zeta1}
\end{table}

\section{Bounding $|1/\zeta(s)|$}\label{mainsec2}

Moving from a bound on the logarithmic derivative of $\zeta(s)$ to one on the reciprocal of $\zeta(s)$ is done in the same way as \cite[Lemma~10]{hiaryleongyangArxiv}. In essence, a bound on the latter depends on a bound for the former in some range of $\Re(s)$. We are able to make some savings by using multiple bounds for $|\zeta'(s)/\zeta(s)|$ across the desired range, using corresponding pairs of $(W,R_1)$ from Tables \ref{tab:log-deriv-zeta} and \ref{tab:log-deriv-zeta1}.

\begin{lemma}\label{1/zeta_main_lem}
Let $t_0\ge e^e$ and $d_1>0$ be constants. Let $\{ W_{j} : 1\leq j\leq J\}$ be a sequence of increasing real numbers where $(W_j, R_{1,j})$ is a pair for which $t\ge t_0$ and
\begin{equation*}
\sigma \ge 1- \frac{\log\log t}{W_j\log t} \quad\text{implies}\quad 
\left| \frac{\zeta'(\sigma +it)}{\zeta(\sigma+it)}\right| \le R_{1,j}\frac{\log t}{\log\log t}.
\end{equation*}
Then, for 
\begin{equation*}
    \sigma \ge 1- \frac{\log\log t}{W_1\log t}
\end{equation*}
we have
\begin{equation*}
\left| \frac{1}{\zeta(\sigma+it)}\right| \le Q_2(d_1,\sigma_1,t_0)\frac{\log t}{\log\log t},
\end{equation*}
where 
\begin{align}\label{Q2}
Q_2(d_1,\sigma_1,t_0)& =  \min_{V}\Biggr\lbrace \max\Biggl\{ \frac{\sigma_1}{\sigma_1-1}\frac{\log\log t_0}{\log t_0},\, V_1(d_1,\sigma_1,t_0), \\
& V(d_1,\sigma_1,t_0)\cdot \exp\Bigg( \sum_{j=1}^{J-1} R_{1,j} \left( \frac{1}{W_j}-\frac{1}{W_{j+1}}\right) + \frac{R_{1,J}}{W_J} +d_1 K_1\Bigg)\Biggr\rbrace 
\Biggr\}, \nonumber
\end{align}
with any $\sigma_1\geq 1+d_1\log\log t_0/\log t_0$, $\gamma$ denoting Euler's constant, $K_1$ defined in Corollary \ref{cor:zeta'/zeta_main_lem1}, $Z$ defined in Lemma \ref{zeta_bound_after_1}, and $V,\,V_1$ defined in \eqref{V min} and Lemma \ref{inverse zeta trig poly}, respectively.
\end{lemma}
\begin{proof}
Let $\delta_1 = \delta_1(t) := d_1\log\log t/ \log t$. We will split the proof into three cases: $\sigma \le 1+\delta_1$, $1+\delta_1 \le \sigma \le \sigma_1$, and $\sigma\ge \sigma_1$. First see that in the range $1- \log\log t/(W_1\log t)\le\sigma \le 1+ \delta_1$ we have 
\begin{align}\label{1/zeta_int}
\log\left| \frac{1}{\zeta(\sigma+it)}\right| 
&= -\textup{Re} \log \zeta(\sigma+it) \nonumber \\
&= - \textup{Re} \log \zeta\left( 1+ \delta_1 +it\right) + \int_\sigma^{1+\delta_1} \textup{Re}\left( \frac{\zeta'}{\zeta}(x+it)\right) \text{d}x. 
\end{align}
Writing $\Delta = \log\log t/\log t$, we can split the integral and rewrite it as
\begin{align*}
\left( \int_{1-\tfrac{\Delta}{W_1}}^{1-\tfrac{\Delta}{W_{2}}} + \cdots + \int_{1-\tfrac{\Delta}{W_j}}^{1-\tfrac{\Delta}{W_{j+1}}} +\cdots + \int_{1-\tfrac{\Delta}{W_J}}^{1} + \int_1^{1+\delta_1} \right) \textup{Re}\left( \frac{\zeta'}{\zeta}(x+it)\right) \text{d}x.
\end{align*}
By Corollary \ref{cor:zeta'/zeta_main_lem1} and the assumption in the lemma we have
\begin{align*}
\int_\sigma^{1+\delta_1} \textup{Re}\left( \frac{\zeta'}{\zeta}(x+it)\right) \text{d}x &\leq \sum_{j=1}^{J-1} R_{1,j} \left( \frac{1}{W_j}-\frac{1}{W_{j+1}}\right) + \frac{R_{1,J}}{W_J} +d_1 K_1.
\end{align*}
We therefore have, from \eqref{1/zeta_int},
\begin{align}\label{1/zeta_int1}
\log\left| \frac{1}{\zeta(\sigma+it)}\right| <& -\log\left| \zeta\left( 1+\frac{d_1\log\log t}{\log t}+ it \right)\right| \\
&+ \sum_{j=1}^{J-1} R_{1,j} \left( \frac{1}{W_j}-\frac{1}{W_{j+1}}\right) + \frac{R_{1,J}}{W_J} +d_1 K_1. \nonumber
\end{align}

To estimate the first term on the right-hand side of \eqref{1/zeta_int1} we take  $\sigma = 1+\delta_1$ and apply \eqref{V min} and Lemma \ref{inverse zeta trig poly}, keeping in mind that we set $\sigma_1\geq 1+\delta_1$. This gives us
\begin{equation*}
\left| \frac{1}{\zeta(1 + \delta_1 +it)}\right| \le V(d_1,\sigma_1,t_0) \frac{\log t}{\log\log t}.
\end{equation*}
We combine this with \eqref{1/zeta_int1} and exponentiate to obtain $Q_2$ for $\sigma \le 1+\delta_1$.

For $1+\delta_1 \le \sigma \le \sigma_1$ we might not have the condition $\sigma \le (1+\gamma)/\gamma$, so we apply \eqref{V1} from Lemma \ref{inverse zeta trig poly} to see that
\begin{equation*}
\left| \frac{1}{\zeta(1 + \delta_1 +it)}\right| \le V_1(d_1,\sigma_1,t_0) \frac{\log t}{\log\log t}.
\end{equation*}
Finally, for $\sigma \ge \sigma_1 >1$ we have from \eqref{trig ineq} and the bound $|\zeta(s)|\le \zeta(\sigma)$ that
    \begin{equation*}
     \left| \frac{1}{\zeta(\sigma+it)}\right| \le \zeta(\sigma) \le \left(\frac{\sigma_1}{\sigma_1-1}\frac{\log\log t_0}{\log t_0}\right)\frac{\log t}{\log\log t}.
    \end{equation*}
Taking the maximum $Q_2$ from all three cases completes the proof.
\end{proof}

Corollary \ref{cor:1/zeta} follows from Lemma \ref{1/zeta_main_lem} by computing $Q_2$ for specific choices of $t_0$ and $W_j$, and optimizing over $d_1$ and $\sigma_1$. These computed values of $Q_2$ are labelled $R_2$.

\begin{proof}[Proof of Corollary \ref{cor:1/zeta}]
    After choosing $t_0$, we aim to minimise $Q_2$ by optimising over $d_1>0$ and $\sigma_1$. This requires choosing $J$, which corresponds to the number of $W_j$ values. Using more $W_j$ will give a better result, but the improvements eventually become negligible, as $Q_2$ is largely determined by the initial few $W_j$.
    
    For computations, we only used values for $W_j$ from Table \ref{tab:log-deriv-zeta} (Corollary \ref{cor:zeta'/zeta_main_lem}). This meant that using $W=70$, for instance, would only be valid for $t\geq 500$. On the plus side, larger $t_0$ meant we could also use smaller $K_1$ from Corollary \ref{cor:zeta'/zeta_main_lem1}. With these factors in mind, we chose to compute two results for a selection of $W$, one for the largest possible range, and the other for the smallest achievable constant. This meant that computing $Q_2$ for $t_0=e^e$ and $W\leq 23$, for instance, only used the entries of Table \ref{tab:log-deriv-zeta} for $t_0=e^e$. Computing $Q_2$ for $t_0 = 500$, however, allowed us to use all available $W$ greater than the one chosen.
    
    For $W=21.24$, we can take $R_2 = Q_2(0.0031,6.52,e^e) = 44910$. Other values of $R_2$ are listed in Table \ref{tab:1/zeta} for the specified $t_0$, and we note that all results for $t_0<500$ had the same approximate optimal $d_1 = 0.0031$ alongside the stated $\sigma_1$, and the results for $t_0=500$ all had $d_1 = 0.0067$ and $\sigma_1 = 12.35$.
\end{proof}

\begin{table}[h]
    \centering
    \begin{tabular}{|c||c|c|c||c|c|}\hline
        $W$ & $t_0$ & $R_2$ & $\sigma_1$ & $t_0$ & $R_2$ \\ \hline
        $21.24$ & $e^e$ & $44910$ & $6.52$ & $500$ & $14978$ \\
        $22$ & $e^e$ & $23227$ & $6.52$ & $500$ & $3438$ \\
        $23$ & $e^e$ & $21453$ & $6.52$ & $500$ & $2494$ \\
        $24$ & $18$ & $13349$ & $7.14$ & $500$ & $2018$ \\
        $25$ & $24$ & $9526$ & $6.94$ & $500$ & $1731$ \\
        $26$ & $33$ & $7332$ & $8.13$ & $500$ & $1532$ \\
        $27$ & $45$ & $5922$ & $8.39$ & $500$ & $1382$ \\ \hline
    \end{tabular}
    \caption{Values of $R_2$ for specific $W$ in Corollary \ref{cor:1/zeta}, for $t\ge t_0$.}
    \label{tab:1/zeta}
\end{table}

If we only consider the case $\sigma\geq 1$ it is possible to reduce $R_2$. This is done in Corollary \ref{cor:1/zeta at 1}. This result is especially useful for bounds on the line $s=1+it$.

\begin{proof}[Proof of Corollary \ref{cor:1/zeta at 1}]
    The proof follows the same method as Lemma \ref{1/zeta_main_lem}, but in the case $\sigma \le 1+\delta_1$, we only need the following bound in \eqref{1/zeta_int}, $$\int_\sigma^{1+\delta_1} \textup{Re}\left( \frac{\zeta'}{\zeta}(x+it)\right) \text{d}x \leq d_1K_1.$$ This simplifies $Q_2(d_1,\sigma_1,t_0)$, replacing the third case in \eqref{Q2} with
    \begin{equation*}
        V_2(d_1,\sigma_1,t_0)\cdot \exp\left( d_1 K_1\right).
    \end{equation*}
    We can now fix $t_0$ and optimise over $d_1$ using the values of $K_1$ in Table \ref{tab:log-deriv-zeta1}.

    For the second assertion of the corollary, Carneiro, Chirre, Helfgott, and Mej{\'i}a-Cordero verify, using interval arithmetic in the proof of Proposition A.2 \cite{carneiro2022optimality}, that $$\left|\frac{1}{\zeta(1+it)}\right| \le 2.079\log t$$ for $2\leq t\leq 500$. This can then be used in conjunction with the result for $t_0 = 500$. All that remains is to check that $$107.7\frac{\log t}{\log\log t} \geq 2.079\log t$$ holds for $3\leq t\leq 500$.
\end{proof}


\begin{table}[h]
    \centering
    \begin{tabular}{|c|c|c|c|}\hline
        $t_0$  & $K_2$ & $d_1$ & $\sigma_1$  \\ \hline
        $e^e$  & $202.3$ & $0.0032$ & $6.64$  \\
        $500$  & $107.7$ & $0.0067$ & $9.80$  \\ 
        $H$  & $103.5$ & $0.0068$ & $9.77$  \\ \hline
    \end{tabular}
    \caption{Values of $K_2$ for specific $t_0$ in Corollary \ref{cor:1/zeta at 1}.}
    \label{tab:1/zeta at 1}
\end{table}

\section{Discussion}

Using the methods of this paper, the main obstruction to obtaining better bounds on $|\zeta'(s)/\zeta(s)|$ and $|1/\zeta(s)|$ is the width of the zero-free region. Lemmas \ref{log_deriv_zeta_lem1} and \ref{log_deriv_zeta_lem2} rely on the Borel--Carath\'eodory theorem (Theorem \ref{borelcaratheodory}), which, in the setting of a general function $f$, already has the best possible constant (see \cite[\S5]{maz2007sharp} for a discussion). Thus, to get improvements, one would have to input additional information on the specific function used. For example, in Lemmas \ref{log_deriv_zeta_lem1} and \ref{log_deriv_zeta_lem2} we use information on the zero-free region of $\zeta(s)$, so one could utilise more information to do better.

Because we sought bounds of the form $\log t/\log\log t$ rather than $\log t$, we were restricted by the bound on $|\zeta(s)|$ in the region $\sigma\ge 1- \omega_2(\log \log t)^2/\log t$. The width of this region hampers us in a similar fashion to the width of the zero-free region, so improving $\omega_2$ is another possibility.

We conclude with a discussion on the trigonometric inequality \eqref{trig ineq} used in the proof of Lemma \ref{inverse zeta trig poly}. When bounding $|1/\zeta(s)|$ for $\sigma$ close to and above 1, \eqref{trig ineq} is superior to the trivial bound $|\zeta(s)|\geq \zeta(2\sigma)/\zeta(\sigma)$ if
\begin{equation}\label{trivial worse}
\frac{\zeta(\sigma)}{\zeta(2\sigma)^4} \ge |\zeta(\sigma+2it)|.
\end{equation}
The validity of this inequality is not obvious, as the left-hand side of \eqref{trivial worse} decreases to $(6/\pi^2)^4\zeta(\sigma)\ge 0.14\zeta(\sigma)$ as $\sigma\to 1$, but the right-hand side can be trivially bounded: $|\zeta(\sigma+2it)|\leq \zeta(\sigma)$. However, if this trivial bound is replaced by a bound depending on $t$, then having sufficiently large $t$ on the right-hand side of \eqref{trivial worse} will prevent the bound from approaching infinity as $\sigma\rightarrow 1$. Indeed, when using Lemma \ref{zeta_bound_after_1} to bound \eqref{trig ineq} over $1<\sigma\le 6.77$, we have
\begin{equation*}
    |\zeta(\sigma+2it)| \le 2.065\left( 1+\frac{\log 2}{e}\right)\frac{\log t}{\log\log t} \le 2.6\frac{\log t}{\log\log t},
\end{equation*}
for $t\geq 2e^e$. In contrast, for $\sigma= 1+\delta_1$ (with $\delta_1$ defined as in the proof of Lemma \ref{1/zeta_main_lem}) we would have
\begin{equation*}
\frac{\zeta(\sigma)}{\zeta(2\sigma)^4} = \frac{\zeta(\sigma)(\sigma-1)}{\zeta(2\sigma)^4 d_1}\frac{\log t}{\log \log t}\ge \frac{(6/\pi^2)^4}{d_1}\frac{\log t}{\log \log t}.
\end{equation*}
Thus \eqref{trivial worse} is true for $d_1\le 0.0525$, and our optimised $d_1$ always fell in this range.

Notice that $Z$ in Lemma \ref{zeta_bound_after_1} depends on, and also tends to, the constant in Theorem \ref{hiary 1 bound} (as $t_0$ gets large). Unfortunately, an increase or decrease in this constant has a minimal impact on our main results, since our main constants are quite large. Thus, even though Lemma \ref{inverse zeta trig poly} offers a bound better than trivial, one would be better served seeking improvements from other avenues like the ones mentioned at the beginning of this section, for more substantial savings.

Another interesting point is that the conditions for trigonometric inequalities like \eqref{trig ineq} to give good bounds on $|1/\zeta(s)|$ for $\sigma>1$ differ from the conditions needed for good zero-free regions. For instance, when dealing with the latter, one requires a non-negative trigonometric polynomial
\begin{equation}\label{trig poly}
   \sum_{n=0}^N a_n\cos(n\theta) = a_0 +a_1\cos (\theta)+\ldots +a_N \cos(N\theta),
\end{equation}
where $a_1>a_0$, with each coefficient $a_n$ non-negative, to state a few conditions (see \cite{Kondrat'ev_1977}, \cite{MossinghoffTrudgian2015}, for more information). However, for $|1/\zeta (s)|$ the criteria
is to have the sum of all coefficients in the polynomial $\sum_{n} |a_n|\le 2a_1$, with $a_1 >0$. This condition is to ensure the overall bound is of the desired order. To illustrate this, if we had used a polynomial in the proof of Lemma \ref{inverse zeta trig poly} with $\sum_{n} |a_n|= 2a_1 +\epsilon$, for any $\epsilon >0$, then we would have ended up with a factor of $(\log t/\log\log t)^{1+(\epsilon/a_1)}$ instead of $\log t/\log\log t$.

The natural question to ask is if one could do better with a different choice of inequality than \eqref{trig ineq} when bounding $|1/\zeta(s)|$. This is equivalent to asking for \eqref{trig poly} with $0 < a_0/a_1< 3/4$, while at the same time satisfying our new criteria. We briefly did a search for such polynomials of higher degree, to no avail. Although there have been many trigonometric polynomials found in the arena of refining zero-free regions, none of these fit our stated new criteria while also improving on $3/4$. We hypothesise that no better polynomials exist. Certainly for a degree-$2$ polynomial of the form
\begin{equation*}
(x+y\cos\theta)^2 = x^2+\frac{y^2}{2}+2xy\cos\theta+\frac{y^2}{2}\cos 2\theta
\end{equation*}
(which is the prequel to \eqref{trig ineq}), it is easily seen that to fulfil our new criteria, one needs $(x-y)^2 \le 0$. Thus the only option is $x=y$, and so \eqref{trig ineq} is the best possible in this situation. 

\clearpage
\bibliographystyle{amsplain} 
\bibliography{references}

\end{document}